\documentclass{amsart}

\usepackage{amsfonts}
\usepackage{amssymb}

\numberwithin{equation}{section}

\theoremstyle{plain}

\newtheorem{theorem}[subsection]{Theorem}
\newtheorem{proposition}[subsection]{Proposition}
\newtheorem{lemma}[subsection]{Lemma}

\theoremstyle{definition}

\renewcommand{\leq}{\leqslant}
\renewcommand{\geq}{\geqslant}
\renewcommand{\P}{\mathbb{P}}

\providecommand{\Sym}{\mathop{\rm Sym}\nolimits}
\providecommand{\supp}{\mathop{\rm supp}\nolimits}

\newcommand{\N}{\mathbb{N}}

\begin{document}

\title{On a non-abelian Balog-Szemer{\'e}di-type lemma}

\author{Tom Sanders}
\address{Department of Pure Mathematics and Mathematical Statistics\\
University of Cambridge\\
Wilberforce Road\\
Cambridge CB3 0WA\\
England } \email{t.sanders@dpmms.cam.ac.uk}

\begin{abstract}
We show that if $G$ is a group and $A \subset G$ is a finite set with $|A^2| \leq K|A|$, then there is a symmetric neighbourhood of the identity $S$ with
\begin{equation*}
S^k \subset A^2A^{-2} \textrm{ and } |S| \geq \exp(-K^{O(k)})|A|.
\end{equation*}
\end{abstract}

\maketitle

\setcounter{section}{1}

Suppose that $G$ is a group and $A \subset G$ is a finite set with doubling $K$, that is $|A^2| \leq K|A|$.  Clearly if $A$ is a collection of free generators then $K=|A|$, but if $K$ is much smaller then it tells us that there must be quite a lot of overlap in the products $aa'$ with $a, a' \in A$.  The extreme instance of this is when $K=1$ and $A$ is necessarily a coset of a subgroup of $A$.  We are interested in the extent to which some sort of structure persists when $K$ is a bit bigger than $1$, say $O(1)$ as $|A| \rightarrow \infty$.
 
 If $G$ is abelian then the structure of $A$ is comprehensively described by the Green-Ruzsa-Fre{\u\i}man theorem \cite{BJGIZR}, but in the non-abelian case no analogue is known.  A number of attempts have been made establishing some remarkable results, see \cite{BJGEB2,BJGEB1,DFNHKIP,EH} and \cite{TCTFrei} for details of these, but a clear description has not yet emerged.  The interested reader may wish to consult \cite{BJG.AGA} for a discussion of the state of affairs.

Fre{\u\i}man-type theorems for abelian groups are applied to great effect throughout additive combinatorics, and many of these applications can make do with a considerably less detailed description of the set $A$.  Moreover, additive combinatorics is now beginning to explore many non-abelian questions and so naturally a Fre{\u\i}man-type theorem in this setting would be very useful.   This is the motivation behind our present work: we want to trade in some of the strength of the description of $A$ in exchange for the increased generality of working in arbitrary groups.  Tao proved a result in this direction in \cite{TCTFrei} for which we require a short definition.  A set $S$ in a (discrete) group $G$ is a \emph{symmetric neighbourhood of the identity} if $1_G \in S$ and $S=S^{-1}$.
\begin{theorem}[{\cite[Proposition C.3]{TCTFrei}}]
Suppose that $G$ is a group, $A \subset G$ is a finite non-empty set such that $|AA^{-1}| \leq K|A|$, and $k \in \N$ and $\epsilon \in (0,1]$ are a pair of parameters.  Then there is a symmetric neighbourhood of the identity $S \subset AA^{-1}$ with $|S| = \Omega_{K,k,\epsilon}(|A|)$ such that for all $l \leq k$
\begin{equation*}
\P(a_1\dots a_{l} \in AA^{-1}| a_1,\dots,a_l \in S) \geq 1-\epsilon.
\end{equation*}
\end{theorem}
The proof uses the celebrated regularity lemma of Szemer{\'e}di and so the resulting bounds are of tower type. 

One would like to remove the $\epsilon$-dependence in Tao's result, but this cannot be done (even in the abelian case, see \cite{IZRAA}) if we are only prepared to accept containment in the two-fold product set $AA^{-1}$.  We shall prove the following $\epsilon$-free result.
\begin{theorem}\label{thm.main}
Suppose that $G$ is a group, $A \subset G$ is a finite non-empty set such that $|A^2| \leq K|A|$, and $k \in \N$ is a parameter.  Then there is a symmetric neighbourhood of the identity $S$ such that
\begin{equation*}
S^k \subset A^2A^{-2} \textrm{ and } |S| \geq  \exp(-K^{O(k)})|A|.
\end{equation*}
\end{theorem}
It should be remarked that in the abelian setting the result follows from Green-Ruzsa modelling and Bogolio{\`u}boff's lemma.  Indeed, this essentially amounts to following the proof of the  Green-Ruzsa-Fre{\u\i}man theorem stopping before the covering argument.  The resulting bound has significantly better $k$ dependence, as it gives $|S| \geq k^{-K^{O(1)}}|A|$.

One of the main applications of Theorem \ref{thm.main} is to produce pairs of sets that are `almost invariant'.  Indeed, if $|A^3| =O(|A|)$ then one can apply the theorem to get a large set $S$ such that 
\begin{equation*}
A \subset S^kA \subset A^2A^{-2}A.
\end{equation*}
By the non-abelian Pl{\"u}nnecke-Ruzsa inequalities of Tao \cite{TCTNC} we have that $|A^2A^{-2}A| = O(|A|)$ and hence by the pigeon-hole principle there is some $l \leq k-1$ such that
\begin{equation*}
|SS^lA| \leq (1+O(1/k))|A| \leq (1+O(1/k))|S^lA|.
\end{equation*}
Writing $A':=S^lA$ we see that the pair $(S,A')$ is almost invariant in the sense that $SA' \approx A'$ with the accuracy of approximation increasing as $k$ increases.

Exactly this argument is given as a `cheat' argument for the proof of \cite[Proposition 5.1]{TCTFrei} where Tao applies \cite[Proposition C.3]{TCTFrei} and first sketches a proof assuming $\epsilon=0$. In view of the above that `cheat' is now sufficient. (In fact this entails a very slight weakening of the conclusion, but the resulting proposition is still more than sufficient for its intended use.)  A similar pigeon-holing argument, but this time on multiple scales is also used in \cite{TSIFEA} on the way to proving a weak non-abelian Fre{\u\i}man-type theorem for so-called multiplicative pairs.

We turn now to the proof of Theorem \ref{thm.main} which uses symmetry sets, popularised in the abelian setting by the book \cite{TCTVHV}.  Suppose that $G$ is a group. Recall that the convolution of two functions $f,g \in \ell^1(G)$ is defined by
\begin{equation*}
f \ast g(x)=\sum_{y \in G}{f(y)g(y^{-1}x)},
\end{equation*}
so that if $A,B \subset G$ then
\begin{equation*}
\supp 1_A \ast 1_B = AB \textrm{ and } 1_A \ast 1_B(x) = |A \cap xB^{-1}|.
\end{equation*}
Given $\eta \in (0,1]$, the \emph{symmetry set of $A$ at threshold $\eta$} is 
\begin{equation*}
\Sym_{\eta}(A):=\{x \in G: 1_A \ast 1_{A^{-1}}(x) \geq \eta |A|\}.
\end{equation*}
It is immediate that $\Sym_\eta(A)$ is a symmetric neighbourhood of the identity contained in $AA^{-1}$, and that we have the nesting property
\begin{equation*}
\Sym_\eta(A) \subset \Sym_{\eta'}(A) \textrm{ whenever } \eta \geq \eta'.
\end{equation*}
A straightforward pigeon-hole argument shows that they also enjoy the following useful sub-multiplicativity property:
\begin{equation*}
\Sym_{1-\epsilon}(A).\Sym_{1-\epsilon'} \subset \Sym_{1-(\epsilon+ \epsilon')}(A)
\end{equation*}
for all $\epsilon,\epsilon' \in [0,1)$ with $\epsilon+\epsilon' < 1$.  See \cite[Lemma 2.33]{TCTVHV} for the abelian details which are exactly the same.

Our main result is the following which provides a plentiful supply of large symmetry sets with threshold close to $1$.
\begin{proposition}\label{prop.intit}
Suppose that $G$ is a group, $A$ is a non-empty subset of $G$ with $|A^2| \leq K|A|$, and $\epsilon \in (0,1]$ is a parameter. Then there is a non-empty set $A'\subset A$ such that
\begin{equation*}
|\Sym_{1-\epsilon}(A'A)| \geq \exp(-K^{O(1/\log (1/(1-\epsilon)))}\log K)|A|.
\end{equation*}
\end{proposition}
One perhaps expects $\epsilon$ to be close to $0$, where $1/\log(1/(1-\epsilon)) = O(\epsilon^{-1})$ is a strong estimate and would simplify the expression above.  However, Tao has pointed out that the result already has content for $\epsilon =1-K^{-\eta}$ and this has been used in the abelian setting in \cite{TSSLS}.  

With this in hand the proof of our main theorem is immediate.
\begin{proof}[Proof of Theorem \ref{thm.main}]
We apply Proposition \ref{prop.intit} with parameter $\epsilon:=1/(k+1)$ to get a non-empty set $A' \subset A$ such that
\begin{equation*}
|\Sym_{1-\epsilon}(A'A)| \geq \exp(-K^{O(k)})|A|.
\end{equation*}
However, by the sub-multiplicativity property of symmetry sets we have
\begin{equation*}
\Sym_{1-\epsilon}(A'A)^k \subset \Sym_{1-k/(k+1)}(A'A) \subset A'A(A'A)^{-1} \subset A^2A^{-2}.
\end{equation*}
The result follows on setting $S:=\Sym_{1-\epsilon}(A'A)$.
\end{proof}
The proof of the proposition involves iterating the following lemma.
\begin{lemma}\label{lem.itlem}
Suppose that $G$ is a group, $A \subset G$ is non-empty and finite, $A' \subset A$ has $|A'| \geq c|A|$ and $|A'A|\leq K|A|$, and $\epsilon \in (0,1]$ is a parameter. Then at least one of the following is true:
\begin{enumerate}
\item there is a subset $A'' \subset A' \subset A$ such that
\begin{equation*}
|A''| \geq c^4|A|/2K \textrm{ and }|A''A| \leq K(1-\epsilon)|A|;
\end{equation*}
\item we have the bound
\begin{equation*}
|\Sym_{1-\epsilon}(A'A)| \geq c^3|A|/2K.
\end{equation*}
\end{enumerate}
\end{lemma}
\begin{proof}
Since $A' \subset A$ we have that $|A'A'| \leq |A'A|$ and
\begin{equation*}
\sum_{x \in G}{1_A \ast 1_{A'}(x)^2} \geq \sum_{x \in G}{1_{A'} \ast 1_{A'}(x)^2} .
\end{equation*}
Now, the Cauchy-Schwarz inequality can be used to bound the right hand side:
\begin{eqnarray*}
\sum_{x \in G}{1_{A'} \ast 1_{A'}(x)^2}\geq \frac{1}{|A'^2|}\left(\sum_{x \in G}{1_{A'} \ast 1_{A'}(x)}\right)^2.
\end{eqnarray*}
However
\begin{equation*}
\sum_{x \in G}{1_{A'} \ast 1_{A'}(x)}=|A'\times A'| = |A'|^2,
\end{equation*}
and so
\begin{equation*}
\sum_{x \in G}{1_{A'} \ast 1_{A'}(x)^2} \geq  |A'|^4/|A'A|.
\end{equation*}
On the other hand for arbitrary sets $B,C,D,E \subset G$ we have
\begin{equation*}
\langle 1_B\ast 1_C, 1_D \ast 1_E\rangle_{\ell^2(G)} = |\{(b,c,d,e) \in B \times C \times D \times E : bc=de\}|,
\end{equation*}
and $bc=de$ if and only if $d^{-1}b=ec^{-1}$ whence
\begin{equation*}
\langle 1_A \ast 1_{A'}, 1_A \ast 1_{A'} \rangle_{\ell^2(G)} = \langle 1_{A^{-1}} \ast 1_A, 1_{A' }\ast 1_{A'^{-1}} \rangle_{\ell^2(G)}.
\end{equation*}
For $t \in G$ write $A'_t:=A' \cap (tA')$ and define 
\begin{equation*}
L:=\{t \in G: |A'_t| \geq |A'|^4/(2|A'A||A|^2)\}.
\end{equation*}
It is easy to check that
\begin{equation*}
|L||A||A'| + \frac{|A'|^4}{2|A'A||A|^2} . |A|^2 \geq \sum_{x \in G}{1_A \ast 1_{A'}(x)^2} ,
\end{equation*}
from which it follows that
\begin{equation*}
|L|\geq |A'|^3/2|A'A||A| \geq c^3|A|/2K
\end{equation*}
since $|A'A| \leq K|A|$ and $|A'| \geq c|A|$. 

Now, if there is some $t \in L$ such that $|A'_tA| \leq (1-\epsilon)|A'A|$ then we terminate in the first case of the lemma with $A''=A'_t$: simply note that $A'' =A'_t \subset A' \subset A \subset G$,
\begin{equation*}
|A''| = |A'_t| \geq \frac{|A'|^4}{2|A'A||A|^2} \geq \frac{c^4}{2K}|A|,
\end{equation*}
since $|A'A| \leq K|A|$ and $|A'| \geq c|A|$, and
\begin{equation*}
|A''A| \leq (1-\epsilon) |A'A| \leq K(1-\epsilon)|A|.
\end{equation*}

In light of this we may assume that there is no such $t \in L$ \emph{i.e.}
\begin{equation*}
|A'_tA| \geq (1-\epsilon)|A'A| \textrm{ for all } t \in L.
\end{equation*}
However, $A'_t A = (A' \cap tA')A \subset (A'A) \cap t(A'A)$, whence
\begin{equation*}
1_{A'A} \ast 1_{(A'A)^{-1}}(t) \geq (1-\epsilon)|A'A| \textrm{ for all }t \in L,
\end{equation*}
and we are in the second case in view of the lower bound on the size of $L$.
\end{proof}
\begin{proof}[Proof of Proposition \ref{prop.intit}]
We apply Lemma \ref{lem.itlem} iteratively to get a sequence of non-empty sets $(A'_i)_{i}$ satisfying
\begin{equation*}
A'_{i+1} \subset A, |A'_{i+1}| \geq |A|/(2K)^{(4^i-1)/3} \textrm{ and }|A'_iA| \leq (1-\epsilon)^iK|A|.
\end{equation*}
First $A'_0:=A$. Now, suppose that we are at stage $i$ of the iteration and apply Lemma \ref{lem.itlem} to the pair $(A'_i,A)$.  If we are in the first case of the lemma then we get a set $A'_{i+1} \subset A$ with
\begin{equation*}
|A'_{i+1}| \geq (1/(2K)^{(4^i-1)/3})^4|A|/2K = |A|/(2K)^{(4^{i+1}-1)/3}
\end{equation*}
and
\begin{equation*}
|A'_{i+1}A| \leq (1-\epsilon)|A'_iA| \leq (1-\epsilon)^{i+1}K|A|.
\end{equation*}
The sequence $(A'_i)_i$ has the desired properties and in light of the last one the iteration certainly terminates at some stage $i_0$ with $i_0 \leq \lceil \log K / -\log (1-\epsilon)\rceil$ since $A'_i$ is non-empty so $|A'_iA| \geq |A|$.

When the iteration terminates we put $A':=A'_i$ and since we are in the second case of Lemma \ref{lem.itlem} we have the desired result.
\end{proof}
It is worth making a number of remarks.  First, a lower bound for $|A'|$ may also be read out of the proof although in applications it is not clear how useful this information is.  The driving observation in the proof of Lemma \ref{lem.itlem} is that
\begin{equation*}
(A'\cap tA')A \subset (A'A)  \cap (tA'A),
\end{equation*}
so if the left hand side is close to $|A'A|$ in size then $t \in \Sym_{1-o(1)}(A'A)$. This rather cute idea comes from the work of Katz and Koester \cite{NHKPK}, where they use it in abelian groups to show that if a set has doubling $K$ then there is a correlating set with larger additive energy than the trivial Cauchy-Schwarz lower bound.

Finally, at about the same time as this paper was produced Croot and Sisask in \cite{ESCOS} developed a different method for analysing sumsets, which turns out to also work for sets of small doubling in non-abelian groups.  Their argument gives a better bound in Theorem \ref{thm.main} showing that one may take $|S| \geq \exp(-O(k^2K\log K))|A|$. 

\section*{Acknowledgements}

The author would like to thank Ben Green for encouraging the writing of this paper, Ben Green and Terry Tao for useful discussions around this topic and the anonymous referee for many useful suggestions.

\bibliographystyle{alpha}

\bibliography{master}

\end{document}